\documentclass[12pt]{amsart}
\usepackage{latexsym, amssymb, amsfonts, amscd}

\theoremstyle{plain}
\newtheorem{Thm}{Theorem}[section]
\newtheorem{Cor}[Thm]{Corollary}

\newtheorem{Prop}[Thm]{Proposition}
\newtheorem{Lem}[Thm]{Lemma}

\newtheorem{Cl}[Thm]{Claim}

\newtheorem{Thm'}{Theorem}[section]
\newtheorem{Cor'}[Thm']{Corollary}
\newtheorem{Prop'}[Thm']{Proposition}
\newtheorem{Lem'}[Thm']{Lemma}
\newtheorem{Cl'}[Thm']{Claim}
\theoremstyle{definition}
\newtheorem{Def}[Thm]{Definition}

\newtheorem{Emp}[Thm]{}

\newtheorem{Not}[Thm]{Notation}

\newtheorem{Def'}[Thm']{Definition}
\newtheorem{Rem'}[Thm']{Remark}
\newtheorem{Rem1'}[Thm']{Remarks}
\newtheorem{Emp'}[Thm']{}
\newtheorem{Ex'}[Thm']{Example}
\newtheorem{Exs'}[Thm']{Examples}
\newtheorem{Con'}[Thm']{Construction}
\newtheorem{Not'}[Thm']{Notation}
\newtheorem{Q'}[Thm']{Question}
\numberwithin{equation}{section}

\newcommand{\ql}{\B{Q}_l}
\newcommand{\qlbar}{\overline{\B{Q}_l}}

\newcommand{\om}{\omega}

\newcommand{\ov}{\overline}
\newcommand{\un}{\underline}

\newcommand{\fq}{\B{F}_q}

\newcommand{\B}[1]{\mathbb#1}
\newcommand{\cal}[1]{\mathcal{#1}}
\newcommand{\form}[1]{(\ref{Eq:#1})}
\newcommand{\C}[1]{\cal#1}

\newcommand{\isom}{\overset {\thicksim}{\to}}

\newcommand{\lra}{\longrightarrow}

\newcommand{\hra}{\hookrightarrow}
\newcommand{\wt}{\widetilde}

\newcommand{\Gm}{\Gamma}

\newcommand{\p}{\partial}

\newcommand{\la}{\lambda}

\newcommand{\rl}[1]{Lemma \ref{L:#1}}

\newcommand{\rcl}[1]{Claim \ref{C:#1}}
\newcommand{\rp}[1]{Proposition \ref{P:#1}}

\newcommand{\re}[1]{\ref{E:#1}}
\newcommand{\rco}[1]{Corollary \ref{C:#1}}

\newcommand{\rt}[1] {Theorem \ref{T:#1}}
\newcommand{\rd}[1]{Definition \ref{D:#1}}
\newcommand{\sm}{\smallsetminus}

\newcommand{\Mat}{\operatorname{Mat}}

\newcommand{\pr}{\operatorname{pr}}

\newcommand{\Spec}{\operatorname{Spec}}

\newcommand{\Tr}{\operatorname{Tr}}
\newcommand{\cl}{\operatorname{cl}}

\newcommand{\Bl}{\operatorname{Bl}}
\newcommand{\Irr}{\operatorname{Irr}}

\newcommand{\ram}{\operatorname{ram}}
\newcommand{\res}{\operatorname{res}}
\newcommand{\Id}{\operatorname{Id}}

\newcommand{\End}{\operatorname{End}}

\begin{document}
%Topmatter

\title[Intersection of a correspondence with a graph of Frobenius]%
{Intersection of a correspondence with a graph of Frobenius}

\author{Yakov Varshavsky}
\address{Institute of Mathematics\\
Hebrew University\\
Givat-Ram, Jerusalem,  91904\\
Israel}
\email{vyakov@math.huji.ac.il }

%\author{Yakov Varshavsky}
%\address{Institute of Mathematics\\
%The Hebrew University of Jerusalem\\
%Givat-Ram, Jerusalem  91904\\
%Israel}
%\email{vyakov@math.huji.ac.il}
\thanks{The work was supported by
THE ISRAEL SCIENCE FOUNDATION (Grant No. 1017/13)}
%\date{February 2006}
%\keywords{Lefschetz trace formula, Deligne's conjecture}
%\subjclass[2000]{Primary: 14F20; Secondary: 11G25, 14G15}
\date{\today}

%\abstract
\begin{abstract}
The goal of this note is to give a short geometric proof of a theorem
of Hrushovski \cite{Hr} asserting that an intersection of a correspondence
with a graph of a sufficiently large power of Frobenius is non-empty.
\end{abstract}
\maketitle

\tableofcontents

\section*{Introduction}

The goal of this work is to give a short geometric proof of the 
following theorem of Hrushovski \cite[Cor 1.2]{Hr}, which has
applications, for example, to algebraic dynamics \cite{Fa}, group theory \cite{BS} and
algebraic geometry \cite{EM}.

Let $\fq$ be a finite field, $\B{F}$ an algebraic closure of $\fq$, and
$X^0$ a scheme of finite type over $\B{F}$, defined over $\fq$. We denote by
$\phi_q=\phi_{X^0,q}:X^0\to X^0$ the geometric Frobenius morphism over $\fq$,
and by $\Gm^0_{q^n}\subset X^0\times X^0$ the graph of
$\phi_{q^n}=(\phi_q)^n$, where the product here and later is taken over
$\B{F}$. Explicitly, $\Gm^0_{q^n}$ is the image of
the morphism $(\Id,\phi_{q^n}):X^0\to X^0\times X^0$. 

\begin{Thm} \label{T:main}
Let $c^0=(c^0_1,c^0_2):C^0\to X^0\times X^0$ be a morphism of
schemes of finite type over $\B{F}$ such that $X^0$ and $C^0$
be irreducible, both $c^0_1$ and $c^0_2$ are dominant, and $X^0$
is defined over $\fq$. 

Then for every sufficiently large $n$,  the preimage
$(c^0)^{-1}(\Gm^0_{q^n})$ is non-empty.
\end{Thm}

\rt{main} has the following corollary.

\begin{Cor} \label{C:sapir}
In the assumptions of \rt{main}, the union $\cup_n
(c^0)^{-1}(\Gm^0_{q^n})$ is Zariski dense in $C^0$.
\end{Cor}

\begin{proof}[Proof of \rco{sapir}]
Let $Z\subset C^0$ be the Zariski closure of $\cup_n
(c^0)^{-1}(\Gm^0_{q^n})\subset C^0$. If $Z\neq C^0$, then $C':=C^0\sm
Z$ is Zariski dense in $C^0$, thus $c^0|_{C'}:C'\to X^0\times X^0$
satisfies all the assumptions of \rt{main}. Hence for every
sufficiently large $n$, we have $(c^0)^{-1}(\Gm^0_{q^n})\cap
C'\neq\emptyset$, contradicting our choice of $Z$.
\end{proof}

Let $f:X^0\to X^0$ be a morphism.
Following Borisov and Sapir \cite{BS}, we call a point $x\in X^0(\B{F})$ to be
{\em $f$-quasi-fixed}, if $f(x)=(\phi_q)^n(x)$ for some $n\in\B{N}$. The following result
follows immediately from \rco{sapir}.

\begin{Cor} \label{C:qfixed}
Let $X^0$ be an irreducible scheme of finite type over $\B{F}$, defined over $\fq$, and let
$f:X^0\to X^0$ be a dominant morphism. Then the set of $f$-quasi-fixed points is Zariski dense.
\end{Cor}

Let $X^0$ be a scheme of finite type over $\B{F}$ and let $f:X^0\to X^0$ be a morphism. We say that $x\in X(\B{F})$ is {\em $f$-periodic}, if $f^m(x)=x$ for some $m\in\B{N}$.
\rco{qfixed} implies the following result (see \cite[Prop 5.5]{Fa}).

\begin{Cor} \label{C:periodic}
Let $X^0$ be a scheme of finite type over $\B{F}$, and let
$f:X^0\to X^0$ be a dominant morphism. Then the set of $f$-periodic points is
Zariski dense.
\end{Cor}
\begin{proof}[Proof of \rco{periodic}]
Since some power $f^m$ stabilizes all irreducible components of $X^0$, we can assume that
$X^0$ is irreducible. Replacing $\fq$ by its finite extension, we can assume that both $X^0$ and $f$ are defined over $\fq$. Then every $f$-quasi-fixed point is $f$-periodic. Indeed,
if $x\in X^0(\B{F}_{q^m})$ satisfies $f(x)=(\phi_q)^n(x)$, then  $f^m(x)=(\phi_q)^{mn}(x)=x$.
Thus the assertion follows from \rco{qfixed}.
\end{proof}

Our proof of \rt{main} goes as follows. First we reduce to the case when
$X^0$ is quasiprojective, $\dim C^0=\dim X^0$, and $c^0$ is a
closed embedding.

Then we choose a compactification $X$ of $X^0$ defined over $\fq$, and
a closed embedding $c=(c_1,c_2):C\to X\times X$, whose restriction
to $X^0\times X^0$ is $c^0$.  We say that $\p X:=X\sm X^0$ is
{\em locally $c$-invariant}, if every point $x\in X$ has an open neighborhood
$U\subset X$ such that $c_2^{-1}(\p X\cap U)\cap
c_1^{-1}(U)\subset c_1^{-1}(\p X\cap U)$.

For every $m\in\B{N}$, we denote by
$c^{(m)}:C\to X\times X$ the map $((\phi_q)^m\circ c_1,c_2)$.
The main step of our argument is to reduce to the situation where
$\p X:=X\sm X^0$ is locally $c^{(m)}$-invariant for all $m$.
Namely, we show that this happens after we replace $X^0$ by its
open subscheme and $X$ by a certain blow-up.

Next, using de Jong theorem on alterations, we can further assume
that $X$ is smooth, and the boundary $\p X$ is a union of smooth
divisors $X_i,i\in I$ with  normal crossings, defined over $\fq$.

Following Pink (\cite{Pi}), we consider the blowup $\wt{Y}:=\Bl_{\cup_I(X_i\times
X_i)}(X\times X)$, and denote by $\pi:\wt{Y}\to X\times X$ the projection map. Let $\Gm_{q^n}\subset X\times X$
be the graph of $\phi_{X,q^n}$, and denote by $\wt{C}\subset\wt{Y}$ and $\wt{\Gm}_{q^n}\subset \wt{Y}$
the strict preimages of $C$ and $\Gm_{q^n}$, respectively.

Replacing $c$ by $c^{(m)}$ for a sufficiently large $m$, we can get to the situation where $\wt{C}\cap\wt{\Gm}_{q^n}\subset\pi^{-1}(X^0\times X^0)$. Then
$\wt{C}\cap\wt{\Gm}_{q^n}=(c^0)^{-1}(\Gm_{q^n}^0)$, so it remains to show that
$\wt{C}\cap\wt{\Gm}_{q^n}\neq\emptyset$. Note that $\wt{Y}$ is smooth, so it
suffices to show that the intersection number
$[\wt{C}]\cdot[\wt{\Gm}_{q^n}]$ is non-zero.

For every subset $J\subset I$, we denote by $X_J$ the intersection
$\cap_{i\in J}X_i$. In particular, $X_{\emptyset}=X$.
For every $i$, the correspondence $c$ induces an endomorphism
$H^i(c)=(c_2)_*\circ (c_1)^*:H^i(X,\ql)\to H^i(X,\ql)$. In particular,
$H^{2d}(c)$ is the
multiplication by $\deg(c_1)\neq 0$. More generally,
for every $J\subset I$ and $i$, the correspondence $c$
induces an endomorphism $H^i(c_J):H^i(X_J,\ql)\to H^i(X_J,\ql)$
(compare \cite[Ch IV]{Laf}).

Then for every $n\in\B{N}$ we have the equality
\begin{equation} \label{Eq:altsum}
[\wt{C}]\cdot[\wt{\Gm}_{q^n}]=\sum_{J\subset
I}(-1)^{|J|}\sum_{i=0}^{2(d-|J|)}(-1)^i\Tr((\phi^*_q)^n\circ H^i(c_J)).
\end{equation}

Choose an embedding $\iota:\ql\hra\B{C}$. By a theorem of
Deligne, every eigenvalue $\la$ of
$\phi^*_q\in\End H^i(X_J,\ql)$ satisfies
$|\iota(\la)|=q^{i/2}$. Therefore the right hand side of
\form{altsum} grows asymptoticically as $\deg(c_1)q^{dn}$, when $n$ is large. In
particular, $[\wt{C}]\cdot[\wt{\Gm}_{q^n}]\neq 0$, when $n$ is sufficiently large.
%\end{Emp}

%Now the assertion of the corollary follows from the theorem of Deligne which asserts that
%for every embedding $\ql\hra\B{C}$, all eigenvalues of the action of
%$\phi^*_{q^n}$ on $H^i(X_J,\ql)$ have absolute values $q^{\frac{ni}{2}}$.

The paper is organized as follows. In the first section we introduce locally invariant subsets 
and show their simple properties. In the second section we show that every correspondence can be made
locally invariant ``near the boundary" after a blowup. In the third section we recall a beautiful geometric construction of Pink, and study its properties. In the fourth section we review basic 
facts about intersection theory and \'etale cohomology and prove a formula for the intersection number. Finally, in the last section we carry out the proof of \rt{main}.

Our proof is essentially self-contained and uses nothing beyond a theorem of de Jong on alterations, standard facts from intersection theory, the Grothendieck-Lefschetz trace formula and purity. Our argument was strongly motivated by the trace formula of Lafforgue \cite[Prop IV.6]{Laf}, which in its turn is based on the construction of Pink.

We thank H\'el\`ene Esnault for her interest and remarks on the first draft of this note.
We also thank Ehud Hrushovski, Mark Sapir, and Luc Illusie for their interest and stimulating conversations.

\section{Locally invariant subsets}

\begin{Not} \label{N:cor} 
Let $k$ be a field. In this work we only will be interested in the case when $k$ is either algebraically closed or finite.

(a) By a {\em correspondence}, we mean a morphism $c=(c_1,c_2):C\to  X\times X$ between schemes of finite type
over $k$. 

(b) For a correspondence $c:C\to X\times X$ and open subsets
$U\subset X$ and $W\subset C$, we denote by $c|_W:W\to X\times X$
and $c|_U:c_1^{-1}(U)\cap c_2^{-1}(U)\to U\times U$ the
restrictions of $c$.

(c) Let $c:C\to X\times X$ and $\wt{c}:\wt{C}\to\wt{X}\times\wt{X}$ be
two correspondences. By a {\em morphism} from $\wt{c}$ to $c$,
we mean a pair of morphisms $[f]=(f,f_C)$, making
the following diagram commutative
\begin{equation} \label{Eq:funct}
\CD
        \wt{X}  @<{\wt{c}_1}<<   \wt{C}       @>{\wt{c}_2}>>         \wt{X}\\
        @V{f}VV                        @V{f_C}VV                       @VV{f}V\\
        X @<{c_1}<<                   C    @>{c_2}>>           X.
\endCD
\end{equation}

(d) Suppose that we are given correspondences $\wt{c}$ and $c$ as in (c) and a morphism $f:\wt{X}\to X$.
We say that $\wt{c}$ {\em lifts} $c$, if there exists a morphism $f_C:\wt{C}\to C$ such that
$[f]=(f,f_C)$ is a morphism from $\wt{c}$ to $c$.
\end{Not}

%\begin{Not} \label{N:closed}
%(a) For a closed subscheme $Z\subset X$, we denote by $\C{I}_Z\subset\C{O}_X$ the sheaf of
%ideals of $Z$, and let
%\end{Not}

\begin{Def} \label{D:inv}
Let $c:Y\to X\times X$ be a correspondence, and let $Z\subset X$ be a
closed subset.

(a) We say that $Z$ is {\em $c$-invariant}, if $c_1(c_2^{-1}(Z))$
is set-theoretically contained in $Z$.

%(b) We say that $Z$ is {\em $c$-invariant in a neighborhood of
%fixed points}, if there exists an open an neighborhood $W\subset
%Y$ of $Fix(c)$ such that $Z$ is $c|_W$-invariant.

(b) We say that $Z$ is {\em locally $c$-invariant}, if for every point
$x\in Z$ there exists an open neighborhood $U\subset X$ of $x$
such that $Z\cap U\subset U$ is
$c|_U$-invariant.
\end{Def}

%\begin{Emp} \label{E:qf}
%{\bf Example.} Let $c:Y\to X\times X$ be a correspondence, and let
%$x\in X$ be a closed point such that $c_2^{-1}(x)$ is finite. Then
%$\{x\}$ is $c$-invariant in a neighborhood of fixed points.
%Indeed, $W:YC\sm [c_2^{-1}(x)\cap c_1^{-1}(X\sm x)]$ is the
%required open subset.
%\end{Emp}

%The following lemma lists simple properties of this notion.

\begin{Lem} \label{L:locinv}
Let $[f]=(f,f_C)$ be a morphism from a correspondence  $\wt{c}:\wt{C}\to\wt{X}\times\wt{X}$ to
$c:C\to X\times X$. If $Z\subset X$ is a locally $c$-invariant
closed subset, then $f^{-1}(Z)\subset\wt{X}$ is locally $\wt{c}$-invariant.

%(a) If $Z$ is locally $c$-invariant, then it is $c$-invariant in a
%neighborhood of fixed points.

%(b) If $Z_1,Z_2\subset X$ are two closed locally $c$-invariant subsets, then the union
%$Z_1\cup Z_2$ is also locally $c$-invariant.
\end{Lem}

\begin{proof}
%(a) Since $X$ is quasi-compact, there exists a finite open
%covering $\cup_i U_i$ of $X$ such that every $Z\cap U_i$ is
%$c|_{U_i}$-invariant. Then $W:=\cup_i c_1^{-1}(U_i)\cap
%c_2^{-1}(U_i)$ is an open neighborhood of $Fix(c)$, and  $Z$ is
%$c|_W$-invariant. Indeed, $c_1^{-1}(U_i)\cap  c_2^{-1}(U_i)\cap
%F(Z)=\emptyset$ for every $i$, hence $W\cap  F(Z)=\emptyset$, as
%claimed.

%(b) For every $x\in X$ there exist an open neighborhood $U$ such that both $Z_1\cap U$ and
% $Z_2\cap U$ are $c|_U$-invariant. Then $(Z_1\cup Z_2)\cap U=(Z_1\cap U)\cup (Z_2\cap U)$
%is $c|_U$-invariant as well.

The assertion is local, therefore we can replace $X$
by an open subset, thus assuming that $Z$ is $c$-invariant. In
this case,  $\wt{c}_2^{-1}(f^{-1}(Z))=f_C^{-1}(c_2^{-1}(Z))$ is set-theoretically contained in
$f_C^{-1}(c_1^{-1}(Z))=\wt{c}_1^{-1}(f^{-1}(Z))$, thus $f^{-1}(Z)$ is $\wt{c}$-invariant.
\end{proof}

\begin{Not} \label{N:fg}
Let $c:Y\to X\times X$ be a correspondence, and $Z\subset X$ a
closed subset. We set $F(c,Z):=c_2^{-1}(Z)\cap c_1^{-1}(X\sm
Z)$, and let $G(c,Z)$ be the union $\cup_{S\in
\Irr(F(c,Z))}[\ov{c_1(S)}\cap \ov{c_2(S)}]\subset X$, where
$\Irr(F(c,Z))$ denotes the set of irreducible components of $F(c,Z)$, and
$\ov{c_i(S)}\subset X$ is the closure of $c_i(S)$.
\end{Not}

\begin{Emp} \label{E:dim}
{\bf Remarks.} (a) Note that $Z$ is $c$-invariant of and only if $F(c,Z)=\emptyset$.
More generally, if $U\subset X$ is an open subset, then  $Z\cap U\subset U$ is $c|_U$-invariant
if and only if $F(c,Z)\cap  c_1^{-1}(U)\cap c_2^{-1}(U)=\emptyset$.

(b) For every $S\in \Irr(F(c,Z))$, we have we have $c_2(S)\subset
Z$, hence $\ov{c_2(S)}\subset Z$. Therefore $G(c,Z)$ is contained
in $Z$.

(c) Note that if $Z_1,Z_2\subset X$ are two closed locally $c$-invariant subsets, then the union
$Z_1\cup Z_2$ is also locally $c$-invariant.

(d) Note that if $Z\subset X$ is locally $c$-invariant, then $Z\cap U\subset U$ is locally $c|_U$-invariant
for every open $U\subset X$.
\end{Emp}

\begin{Lem} \label{L:locinv2}
Let $c:Y\to X\times X$ be a correspondence, and $Z\subset X$ a
closed subset. Then $X\sm G(c,Z)\subset X$ is the largest open
subset $U\subset X$ such that $Z\cap U\subset U$ is locally
$c|_U$-invariant.
\end{Lem}

\begin{proof}
Let $U\subset X$ be an open subset. By \re{dim}(a), $Z\cap U\subset U$ is $c|_U$-invariant
if and only if $S\cap c_1^{-1}(U)\cap c_2^{-1}(U)=\emptyset$ for every $S\in \Irr(F(c,Z))$.
Since $S$ is irreducible, this happens if and only if either $S\cap  c_1^{-1}(U)=\emptyset$
or $S\cap c_2^{-1}(U)=\emptyset$. But the condition $S\cap
c_i^{-1}(U)=\emptyset$ is equivalent to $c_i(S)\cap  U=\emptyset$
and hence to $U\subset X\sm \ov{c_i(S)}$.

By the proven above, a point $x\in X$ has
an open neighbourhood $U\subset X$ such that $Z\cap U\subset U$ is $c|_U$-invariant
if and only if for every  $S\in \Irr(F(c,Z))$, we have
$x\notin\ov{c_1(S)}$ or $x\notin\ov{c_2(S)}$. Therefore this happens if and only if
$x$ does not belong to $\cup_{S\in \Irr(F(c,Z))}[\ov{c_1(S)}\cap
\ov{c_2(S)}]=G(c,Z)$.
\end{proof}

\begin{Cor} \label{C:locinv}
Let $c:Y\to X\times X$ be a correspondence.

(a) A closed subset $Z\subset X$ is locally $c$-invariant if and only if $G(c,Z)=\emptyset$.

(b) For two closed subsets $Z_1,Z_2\subset X$, we have $G(c,
Z_1\cup Z_2)\subset G(c, Z_1)\cup G(c, Z_2)$.

(c) If $c_2$ is quasi-finite, then $\dim G(c,Z)\leq\dim Z-1$, where we set
$\dim\emptyset:=-\infty$.
\end{Cor}

\begin{proof}
(a) follows immediately from the lemma.

(b) Set $U_i:=X\sm G(c,Z_i)$ and $U:=U_1\cap U_2$. Then every
$Z_i\cap U_i\subset U_i$ is locally $c|_{U_i}$-invariant by the lemma, hence  $(Z_1\cup
Z_2)\cap U\subset U$ is locally $c|_U$-invariant  by remarks \re{dim}(c), (d).
Hence $G(c, Z_1\cup Z_2)\subset X\sm U=  G(c, Z_1)\cup G(c,
Z_2)$ by the lemma.

(c) We may assume that $Z$ is non-empty and irreducible (using (b)). We want to show that
for every $S\in \Irr(F(c,Z))$, we have $\dim(\ov{c_1(S)}\cap\ov{c_2(S)})<\dim Z$.
Since $S\subset c_2^{-1}(Z)$ (see \re{dim}(b)) and $c_2$ is quasi-finite,
we have $\dim S\leq\dim Z$. Since $c_1(S)\cap \ov{c_2(S)}\subset
(X\sm Z)\cap Z=\emptyset$ (see \re{dim}(b)), we conclude that
$\ov{c_1(S)}\cap\ov{c_2(S)}\subset\ov{c_1(S)}\sm c_1(S)$, thus
$\dim(\ov{c_1(S)}\cap \ov{c_2(S)})<\dim S\leq\dim Z$.
\end{proof}

%\begin{Not}
%For a correspondence $c:Y\to X\times X$ and a dominant generically finite morphism
%$\pi:\wt{X}\to X$,
%we denote by $\wt{Y}\subset Y\times_{X\times X}\wt{X}\times \wt{X}$ the union of irreducible
%components, which are dominant over $c$ and denote by $\wt{c}:\wt{Y}\to\wt{X}\times \wt{X}$ the
%projection map. The correspondence $\wt{c}$ is called the strict preimage of $c$ under $\pi$.
%\end{Not}

\begin{Def} \label{D:contr}
Let $c:C\to X\times X$ be a correspondence, and let $Z\subset X$ be a
closed subset, or, what is the same, a closed reduced subscheme.

(a) Denote by $\C{I}_Z\subset\C{O}_X$ the sheaf of ideals of $Z$, and 
let $c_i^{\cdot}(\C{I}_Z)=c_i^{-1}(\C{I}_Z)\subset \C{O}_C$ be its inverse image (as a sheaf of sets).

(b) Following \cite{Va}, we say that $c$ is {\em contracting near
$Z$}, if $c_1^{\cdot}(\C{I}_{Z})\subset c_2^{\cdot}(\C{I}_{Z})\cdot\C{O}_C$, and there exists $n>0$
such that $c_1^{\cdot}(\C{I}_{Z})^n\subset
c_2^{\cdot}(\C{I}_{Z})^{n+1}\cdot\C{O}_{C}$.

(c) We say that $c$ is {\em locally contracting near $Z$}, if for
every $x\in X$ there exists an open neighborhood $U\subset X$ of
$x$ such that $c|_U$ is contracting near $Z\cap U$.
\end{Def}

\begin{Emp} \label{E:finfld}
{\bf Correspondences over finite fields.} Let $c:C\to X\times X$ be a correspondence over $\B{F}$ such that $X$ is defined over $\fq$, and let $Z\subset X$ be a closed subset. 

(a) For $n\in\B{N}$, we denote by $c^{(n)}:C\to X\times X$ the correspondence $(\phi_{q^n}\circ c_1,c_2)$.

(b) We say that $Z$ is {\em locally $c$-invariant over $\fq$},  
if the open neighborhood $U$ of $x$ from \rd{inv}(b) can be chosen to be defined over $\fq$.  

(c) We say that $c$ is {\em locally contracting near $Z$ over $\fq$}, 
if the open neighborhood $U$ from \rd{contr}(c) can be chosen to be defined over $\fq$.  

%We also denote by  $\varphi_{q^n}=(\varphi_q)^n:X\isom X$ be the inverse of the arithmetic Frobenius isomorphism
%over $\B{F}_{q^n}$.
\end{Emp}

\begin{Emp} \label{E:remfinfld}
{\bf Remark.} Note that if $Z$ is defined over $\fq$ and locally $c$-invariant over $\fq$ (see \re{finfld}(b)), then $Z$ is locally $c^{(n)}$-invariant over $\fq$ for every $n\in\B{N}$. Indeed, we immediately reduce to the case when $Z$ is $c$-invariant.  Then
$\phi_{q^n}(Z)\subset Z$, thus $Z$ is $c^{(n)}$-invariant. 

%(b) In the situation of \re{finfld}(b), let $C$ and $c$ be defined over $\fq$. Then $Z$ is $\fq$-locally $c$-invariant 
%if and only if $Z$ is locally $c$-invariant, and $c$ is $\fq$-locally contracting near $Z$ if and only if 
%$c$ is locally contracting near $Z$. 
\end{Emp}

\begin{Emp} \label{E:degree}
{\bf The ramification degree.} Let $f:Y\to X$ be a morphism of Noetherian schemes, and let $Z$ be a closed subset of $X$.
Let $f^{-1}(Z)\subset Y$ be the schematic preimage of $Z$, and let $\ram(f,Z)$ be the smallest positive integer $m$ such that
$\sqrt{\C{I}_{f^{-1}(Z)}}^m\subset\C{I}_{f^{-1}(Z)}$.
\end{Emp}

The following lemma and its proof are basically copied from \cite[Lem 2.2.3]{Va}.

\begin{Lem} \label{L:contr}
In the situation of \re{finfld}, let $n\in\B{N}$ be such that
$q^n>\ram(c_2,Z)$ and $Z$ is $c^{(n)}$-invariant. Then the correspondence $c^{(n)}$ is contracting near $Z$.
\end{Lem}

\begin{proof}
Set $m:=\ram(c_2,Z)$, and let $\varphi_{q^n}$ be the inverse of the arithmetic Frobenius isomorphism
$X\isom X$ over $\B{F}_{q^n}$. Then for every section $f$ of $\C{O}_{X}$, we have
$\phi_{q^n}^{\cdot}(f)=(\varphi_{q^n})^{\cdot}(f)^{q^n}$.
Therefore $(c_1^{(n)})^{\cdot}(\C{I}_Z)$ equals
$c_1^{\cdot}(\varphi_{q^n})^{\cdot}(\C{I}_Z)^{q^n}$.

Since $Z$ is $c^{(n)}$-invariant, we get an inclusion $\C{I}_{(c_1^{(n)})^{-1}(Z)}\subset\sqrt{\C{I}_{c_2^{-1}(Z)}}$. Hence
$c_1^{\cdot}(\varphi_{q^n})^{\cdot}(\C{I}_Z)
\subset\sqrt{\C{I}_{c_2^{-1}(Z)}}$, thus
$(c_1^{(n)})^{\cdot}(\C{I}_Z)\subset\sqrt{\C{I}_{c_2^{-1}(Z)}}^{q^n}$.
As  $q^n\geq m+1$, we conclude that
$(c_1^{(n)})^{\cdot}(\C{I}_Z)\subset \sqrt{\C{I}_{c_2^{-1}(Z)}}^{m+1}
\subset \C{I}_{c_2^{-1}(Z)}$. Furthermore, $(c_1^{(n)})^{\cdot}(\C{I}_Z)^m$ is contained in
$\sqrt{\C{I}_{c_2^{-1}(Z)}}^{m(m+1)}
\subset(\C{I}_{c_2^{-1}(Z)})^{m+1}$. Hence $c^{(n)}$ is contracting near $Z$, as claimed.
\end{proof}

\begin{Cor} \label{C:contr} 
In the situation of \re{finfld}(b), let $n\in\B{N}$ be such that
$q^n>\ram(c_2,Z)$ and $Z$ is locally $c^{(n)}$-invariant over $\fq$. 
Then the correspondence $c^{(n)}$ is locally contracting near $Z$ over $\fq$.
\end{Cor}

\begin{proof}
For every open subset $U\subset X$ we have $\ram((c|_U)_2,Z\cap U)\leq \ram(c_2,Z)$. Thus the assertion follows 
from  \rl{contr}. 
\end{proof}
\section{Main technical result}

%The following lemma is crucial for what follows.
\begin{Emp} \label{E:setup}
{\bf Set up.} Let  $c:C\to X\times X$ be a correspondence over $k$, and let $X^0\subset X$ be a non-empty open subset
such that $X$ and $C$ are irreducible, $c_2$ is dominant, and $\dim C=\dim X$.
\end{Emp}

\begin{Lem} \label{L:quasifinite}
In the situation of \re{setup}, there exist non-empty open subsets $V\subset U\subset X^0$ such
that

\indent\indent(i) $c_1^{-1}(V)\subset c_2^{-1}(U)$;

\indent\indent(ii) the closed subset $U\sm V\subset U$ is locally
$c|_{U}$-invariant.
\end{Lem}

\begin{proof}
Since $\dim C=\dim X$ and $c_2$ is dominant, there exists a non-empty open subset
$U_0\subset X^0$ such that $c_2|_{c_2^{-1}(U_0)}$ is quasi-finite.
By induction, we define for every $j\geq 0$ open subsets
$V_j\subset U_j\subset X^0$ by the rules
\begin{equation} \label{Eq:vj}
V_j:=U_j\sm \ov{c_1(c_2^{-1}(X\sm U_j))},
\end{equation}
\begin{equation} \label{Eq:zj}
Z_{j}:=G(c|_{U_{j}},U_{j}\sm V_{j}), \text{ and }U_{j+1}:=
U_{j}\sm Z_{j}\subset U_j.
\end{equation}

First we claim that $U_j$ and $V_j$ are non-empty. Indeed, $U_0\neq\emptyset$ by construction, and if
$U_j\neq\emptyset$, then $c_2^{-1}(X\sm U_j)\neq C$ since $c_2$ is dominant, hence
$\dim c_2^{-1}(X\sm U_j)<\dim C$, since $C$ is irreducible. Thus
\[\dim\ov{c_1(c_2^{-1}(X\sm U_j))}\leq \dim c_2^{-1}(X\sm U_j)<\dim C=\dim X=\dim U_j,
\]
hence $V_j\neq\emptyset$.
Finally, $Z_j\subset U_{j}\sm V_{j}$ (see \re{dim}(b)), thus $U_{j+1}\supset V_j$, hence
$U_{j+1}\neq\emptyset$.

We claim that for every $j\geq 0$, we have

(i)$'$ $c_1^{-1}(V_j)\subset c_2^{-1}(U_j)$,  and

(ii)$'$  $U_{j+1}\sm V_{j}=U_{j+1}\cap (U_j\sm V_j)$ is locally
$c|_{U_{j+1}}$-invariant.

\noindent Indeed,  (i)$'$ is equivalent to the equality $c_1^{-1}(V_j)\cap c_2^{-1}(X\sm U_j)=\emptyset$, hence to the equality
$V_j\cap c_1(c_2^{-1}(X\sm U_j))=\emptyset$, so (i)$'$ follows from \form{vj}. Next, (ii)$'$ follows from
\rl{locinv2}.

We will show that $Z_j=\emptyset$ for some $j$.
In this case, $U:=U_j=U_{j+1}$ and $V:=V_j$ satisfy the properties of the
lemma. Indeed, properties (i) and (ii) would follow from (i)$'$ and (ii)$'$,
respectively.

Since $U_{j+1}\subset U_j$, formula \form{vj} implies that $V_{j+1}\subset V_j$. It suffices to show that for every $j$ we have inequalities
\begin{equation} \label{Eq:dim}
\dim Z_{j+1}+1\leq\dim(V_{j}\sm V_{j+1})\leq\dim Z_j.
\end{equation}

Let $\ov{V_{j}\sm V_{j+1}}\subset U_{j+1}$ be the closure of $V_{j}\sm V_{j+1}$. Then
$U_{j+1}\sm V_{j+1}$ equals $(U_{j+1}\sm V_{j})\cup(V_{j}\sm V_{j+1})=(U_{j+1}\sm V_{j})\cup\ov{V_{j}\sm V_{j+1}}$, hence we conclude from \rco{locinv}(b) that  $Z_{j+1}=G(c|_{U_{j+1}},U_{j+1}\sm V_{j+1})$
is contained in
\[
G(c|_{U_{j+1}},\ov{V_{j}\sm V_{j+1}})\cup G(c|_{U_{j+1}},U_{j+1}\sm V_{j}).
\]
Using (ii)$'$ and \rco{locinv}(a), we conclude that
$G(c|_{U_{j+1}},U_{j+1}\sm V_{j})$ is empty, thus
$Z_{j+1}\subset G(c|_{U_{j+1}},\ov{V_{j}\sm V_{j+1}})$. Therefore, by
\rco{locinv}(c), we get inequalities
$\dim Z_{j+1}+1\leq \dim G(c|_{U_{j+1}},\ov{V_{j}\sm V_{j+1}})+1\leq\dim\ov{V_{j}\sm
V_{j+1}}=\dim(V_{j}\sm V_{j+1})$.

Next, since  $X\sm U_{j+1}=(X\sm U_j)\cup Z_j$, it follows from
\form{vj} and \form{zj} that $V_{j+1}=V_{j}\sm[Z_j\cup \ov{c_1(c_2^{-1}(Z_{j}))}]$, thus
$V_{j+1}\sm V_{j}\subset Z_j\cup \ov{c_1(c_2^{-1}(Z_{j}))}$.

Finally, since $Z_{j}\subset U_{j}\subset U_0$, we conclude that
$c_2|_{c_2^{-1}(Z_{j})}$ is quasi-finite, hence
$\dim \ov{c_1(c_2^{-1}(Z_{j}))}\leq\dim c_2^{-1}(Z_{j})\leq \dim
Z_{j}$. Therefore $\dim(V_{j}\sm V_{j+1})\leq\dim Z_j$, and the proof of \form{dim} is complete.
\end{proof}

\begin{Prop} \label{P:blowup}
In the situation of \re{setup}, there exists a non-empty open subset
$V\subset X^0$ and a blow-up $\pi:\wt{X}\to X$, which is an
isomorphism over $V$, such that for every correspondence
$\wt{c}:\wt{C}\to \wt{X}\times\wt{X}$ lifting $c$, the closed
subset $\wt{X}\sm \pi^{-1}(V)\subset \wt{X}$ is locally
$\wt{c}$-invariant.
\end{Prop}

\begin{proof}
The argument goes similarly to that of \cite[Lem.
1.5.4]{Va}. Let $V\subset U\subset X^0$ be as in \rl{quasifinite}. Set
$F:=F(c,X\sm V)=c_2^{-1}(X\sm V)\cap c_1^{-1}(V)$. For every $S\in
\Irr(F)$, we denote by $\C{K}_S$ the sheaf of ideals
$\C{I}_{\ov{c_1(Z)}}+\C{I}_{\ov{c_2(Z)}}\subset \C{O}_X$ of the schematic intersection $\ov{c_1(S)}\cap
\ov{c_2(S)}$, and set $\C{K}:=\prod_{S\in \Irr(F)}\C{K}_S\subset\C{O}_X$. Let
$\wt{X}$ be the blow-up $\Bl_{\C{K}}(X)$, and denote by
$\pi:\wt{X}\to X$ the canonical projection.

We claim that $V$ and $\pi$ satisfy the required properties.
Notice that the support of $\C{O}_X/\C{K}$
equals  $\cup_S(\ov{c_1(S)}\cap
\ov{c_2(S)})=G(c,X\sm V)$. Since $U\sm V$ is locally
$c|_U$-invariant, $G(c,X\sm V)$ is therefore contained
in $X\sm U$ (by \rl{locinv2}). In particular, $\pi$ is an
isomorphism over $U$, hence over $V\subset U$.

Next, we show that for every $S\in \Irr(F)$ we have
$\ov{\pi^{-1}(c_1(S))}\cap \ov{\pi^{-1}(c_2(S))}=\emptyset$. By the definition of
$\pi:\wt{X}\to X$, the strict preimages of $\ov{c_1(S)}$ and
and $\ov{c_2(S)}$ in $\wt{X}$ do not intersect. Thus it suffices to show that every
$\ov{\pi^{-1}(c_i(S))}$ is a strict preimage of $\ov{c_i(S)}$. Since $\pi$ is an isomorphism over $U$,
it suffices to show that both $c_1(S)$ and $c_2(S)$ are contained in $U$.
But $S$ is contained in $F\subset c_1^{-1}(V)\subset c_2^{-1}(U)$ (by
\rl{quasifinite}(i)), hence $c_1(S)\subset V\subset U$  and
$c_2(S)\subset U$.

Now we are ready to show the assertion. Let
$\wt{c}:\wt{C}\to\wt{X}\times\wt{X}$ be any correspondence lifting
$c$, and denote by $\pi_C$ the corresponding morphism $\wt{C}\to
C$. We claim that $\wt{Z}:=\wt{X}\sm \pi^{-1}(V)=\pi^{-1}(X\sm V)$
is locally $\wt{c}$-invariant. Set $\wt{F}:=F(\wt{c},\wt{Z})$ and
fix $\wt{S}\in \Irr(\wt{F})$. We want to show that
$\ov{\wt{c}_1(\wt{S})}\cap \ov{\wt{c}_1(\wt{S})}=\emptyset$ (use \rco{locinv}(a)).
Observe that $\wt{F}$ equals
\[
\wt{c}_2^{-1}(\pi^{-1}(X\sm V))\cap\wt{c}_1^{-1}(\pi^{-1}(V))=
\pi_C^{-1}(c_2^{-1}(X\sm V)\cap c_1^{-1}(V))=\pi_C^{-1}(F).
\]
Therefore $\pi_C(\wt{F})$ is contained in $F$. Hence there exists
$S\in \Irr(F)$ such that $\pi_C(\wt{S})\subset S$. Then for every
$i=1,2$, we have $\pi(\wt{c}_i(\wt{S}))=c_i(\pi_C(\wt{S}))\subset
c_i(S)$. Thus $\wt{c}_i(\wt{S})\subset\pi^{-1}(c_i(S))$. Hence the
intersection $\ov{\wt{c}_1(\wt{S})}\cap \ov{\wt{c}_1(\wt{S})}$ is
contained in $\ov{\pi^{-1}(c_1(S))}\cap
\ov{\pi^{-1}(c_2(S))}=\emptyset$. Therefore $\ov{\wt{c}_1(\wt{S})}\cap \ov{\wt{c}_1(\wt{S})}=\emptyset$, as
claimed.
\end{proof}

For the applications, we will need the following version of \rp{blowup}.

%\begin{Cl} \label{C:blowup}
%In the situation of \re{setup}, assume that $k=\B{F}$, and that $X$ and $X^0$ are defined over $\fq$.
%Then an open subsets $V\subset X^0$ and the blow-up $\pi:\wt{X}\to X$ from \rp{blowup} can be assumed to defined over $\fq$.
%\end{Cl}

%\begin{proof}
%First we claim that open subsets $V\subset U\subset X^0$ from \rl{quasifinite} can be assumed to defined over $\fq$.

%Namely, to show it we modify the argument of \rl{quasifinite} as follows. First of all, the open subset
%$U_0\subset X^0$ can be chosen to be defined over $\fq$. Then, by induction, we define for every $j\geq 0$ open subsets
%$V_j\subset U_j\subset X^0$ by the formulas
%$V_j:=U_j\sm \cup_{n}\varphi_q^n(\ov{c_1(c_2^{-1}(X\sm U_j))})$, $Z_{j}:=G(c|_{U_{j}},U_{j}\sm V_{j})$, and $U_{j+1}:=
%U_{j}\sm \cup_{n}\varphi_q^n(Z_{j})\subset U_j$.

%By construction, $U_j$ and $V_j$ are non-empty, defined over $\fq$ and satisfy (i)$'$ and (ii)$'$.

%TO FINISH!!!

%\end{proof}

%\begin{Not} \label{N:fq}
%For a scheme $X$ over $\fq$, we denote by $\Fr_q:X\to X$ the absolute Frobenius over $\fq$.
%For a correspondence $c:C\to X\times X$ and $n\in\B{N}$, we denote by $c^{(n)}:C\to X\times X$ its Frobenius twist
%$((\Fr_q)^n\circ c_1,c_2)$.
%\end{Not}

\begin{Cor} \label{C:blowup}
In the situation of \re{setup}, assume that $k=\B{F}$, and that $X$ and $X^0$ are defined over $\fq$.

Then there exists an open subset $V\subset X^0$ and a blow-up $\pi:\wt{X}\to X$ which is an isomorphism over $V$,
such that both $V$ and $\pi$ are defined over $\fq$, and for every
map $\wt{c}:\wt{C}\to \wt{X}\times \wt{X}$ lifting $c$ the closed subset $\wt{X}\sm
\pi^{-1}(V)\subset \wt{X}$ is locally $\wt{c}$-invariant over $\fq$.
\end{Cor}

\begin{proof}
By assumption, there exists a scheme of finite type $\un{X}$ over $\fq$ and an open subscheme $\un{X}^0$ of $\un{X}$, whose base change to
$\B{F}$ are $X$ and $X^0$, respectively. Let $\om:X\to\un{X}$ be the canonical morphism.

Choose $r\in\B{N}$ such that $C$ and $c:C\to X\times X$ are defined over $\B{F}_{q^r}$. Then there exists $r\in\B{N}$ and a 
scheme of finite type $\un{C}$ over $\B{F}_{q^r}$, whose base change to $\B{F}$ is $C$, such that the composition $(\om\times\om)\circ c:C\to \un{X}\times\un{X}$ factors through $\un{c}:\un{C}\to \un{X}\times\un{X}$. Then $\un{c}$ satisfies all the assumptions of \re{setup} for $k=\fq$.

Let $\un{V}\subset\un{X}^0$ and $\un{\pi}:\un{\wt{X}}\to\un{X}$ be an open subset and a blow-up 
from \rp{blowup}, respectably, and let $V\subset X^0$ and
$\pi:\wt{X}\to X$ be their base changes to $\B{F}$. Then $V$ and $\pi$ satisfy the required properties.

Indeed, since $\wt{c}$ lifts $c$, the composition $\wt{d}:=(\om\times\om)\circ\wt{c}:\wt{C}\to\un{\wt{X}}\times\un{\wt{X}}$ lifts $\un{c}$. 
Hence, by the assumption on $\un{V}$ and $\un{\pi}$, the closed subset $\un{\wt{X}}\sm \un{p}^{-1}(\un{V})\subset \un{\wt{X}}$ is locally $\wt{d}$-invariant. Finally, since $\wt{c}$ is a lift of $\wt{d}$, we conclude as in \rl{locinv} that $\wt{X}\sm\pi^{-1}(V)\subset \wt{X}$ is locally $\wt{c}$-invariant over $\fq$.
\end{proof}

%\begin{Rem} \label{R:alg}
%$\pi$ is constructed by an explicit algorithm.
%\end{Rem}

%\begin{Cor} \label{C:dejong}
% Let $c:Y\to X\times X$ be a correspondence such that $Y$ and $X$ are irreducible,
%$c_2$ is dominant and generically finite. Then for every  non-empty open subset $X^0\subset X$,
%there exists a non-empty open subset $V\subset X^0$ and a proper generically \'etale morphism
%$\pi:\wt{X}\to X$ satisfying the following properties:
%
%(i)   $c_2|_{c_2^{-1}(V)}$ is quasi-finite;
%
%(ii) $\wt{X}$ is smooth and $\pi^{-1}(X\sm V)$ is a string divisor with normal crossings;
%
%(iii) for every correspondence $\wt{c}:\wt{Y}\to \ov{X}\times\ov{X}$ lifting $c$,
%the closed subset $\pi^{-1}(X\sm V)\subset \wt{X}$ is locally $\wt{c}$-invariant.
%\end{Cor}
%
%\begin{proof}
%Let  $V\subset X^0$ and  $\pi':X'\to X$ be the open subset and the blow-up constructed
%in the proposition. By the theorem of de Jong, there exists a proper generically \'etale morphism
%$\pi'':\wt{X}\to X'$ such that $\wt{X}$ is smooth, while $\pi''^{-1}(\pi'^{-1}(X\sm V))$
%is a strict divisor with normal crossings.
%We claim that $V$ and the composition $\pi=\pi'\circ\pi'':\wt{X}\to X$
%satisfy the condition (i)-(iii). Indeed, (i) and (ii) follows by the construction. To show
%(iii), notice that correspondence $\wt{c}$ lifts the correspondence
%$c':Y\times_{X\times X}(X'\times X')\to X'\times X'$ obtained from $c$ by the base change.
%Then $\pi'^{-1}(X\sm V)\subset X'$ is locally $c'$-invariant by the proposition, hence
%$\pi^{-1}(X\sm V)\subset \wt{X}$ is locally $\wt{c}$-invariant by \rl{locinv}(c).
%\end{proof}
\pagebreak
\section{Geometric construction of Pink \cite{Pi}}

\begin{Emp} \label{E:pink}
{\bf The construction}. 

(a) Let $X$ be a smooth scheme of relative
dimension $d$ over a field $k$, and let  $X_{i}\subset X$, $i\in I$ be a
finite collection of smooth divisors with normal crossings.
We set $\p X:= \cup_{i\in I}X_{i}\subset X$ and $X^0:=X\sm \p X$.

(b) For every $J\subset I$, we set $X_J:=\cap_{i\in J} X_{i}$. In particular, $X_{\emptyset}=X$.
Then every $X_{J}$ is either empty, or smooth over $k$ of relative dimension
$d-|J|$. We also set $\p X_J:=\cup_{i\in I\sm J}X_{J\cup\{i\}}$, and
$X_J^0:=X_J\sm \p X_J$.

(c) Set $Y:=X\times X$, $Y_i:=X_i\times X_i$, $i\in I$, $\p Y:=\cup_{i\in I} Y_i$ and $Y^0:=Y\sm \p
Y$. We denote by $\C{K}_i:=\C{I}_{Y_i}\subset\C{O}_Y$ the sheaf
of ideals of $Y_i$, set $\C{K}:=\prod_{i\in I}\C{K}_i$, let $\wt{Y}$ be the blow-up $\Bl_{\C{K}}(Y)$, and let
$\pi:\wt{Y}\to Y$ be the projection map. Then $\pi$ is an isomorphism over $Y^0$.

(d) For every  $J\subset I$, we set $Y_J:=X_J\times X_J\subset Y$ and 
$E_J:=\pi^{-1}(Y_J)\subset \wt{Y}$, denote by $i_J$ the inclusion
$E_J\hra\wt{Y}$, and by $\pi_J$ the projection $E_J\to Y_{J}$. We
also set $\p Y_J:=\cup_{i\in I\sm J}Y_{J\cup\{i\}}$ and
$Y_J^0:=Y_J\sm \p Y_J$.  Explicitly, a point $y\in Y$ belongs to $Y_J^0$
 if and only if $y\in Y_j$ for every $j\in J$, and
$y\notin Y_j$ for every $j\in I\sm J$.
%and $E_J^0:=p_J^{-1}(Y_J^0)$.
\end{Emp}

\begin{Emp} \label{E:basic}
{\bf Basic case.} (a) Assume that $X=\B{A}^I$ with coordinates
$\{x_i\}_{i\in I}$,  and let $X_i=Z(x_i)\subset
X$ (the zero scheme of $x_i$) for all $i\in I$.

(b) The product $Y=X\times X$ is the affine space
$(\B{A}^2)^I$ with coordinates
$x_i,y_i,i\in I$, and $\wt{Y}=(\wt{\B{A}^2})^m$,
where $\wt{\B{A}^2}:=\Bl_{(0,0)}(\B{A}^2)$. Explicitly,  $\wt{Y}$ is
a closed subscheme of the product $(\B{A}^2\times\B{P}^1)^I$ with
coordinates $(x_i,y_i,(a_i:b_i))_{i\in I}$ given by equations
$x_ib_i=y_ia_i$.

(c) For every $J\subset I$, the subschemes $Y_J\subset Y$ and
$E_J\subset\wt{Y}$ are given by equations $x_j=y_j=0$ for all $j\in
J$. Thus $Y_J\cong(\B{A}^2)^{I\sm J}$, $E_J\cong
(\wt{\B{A}^2})^{I-J}\times(\B{P}^1)^{J}$, and $\pi_J:E_J\to Y_J$ is the
projection $(\wt{\B{A}^2})^{I\sm J}\times(\B{P}^1)^{J}\to
(\wt{\B{A}^2})^{I\sm J}\to (\B{A}^2)^{I\sm J}$.

%(d) Note that $Y^0_J\subset Y_J$ and $E^0_J\subset E_J$ are given by
%inequalities $x_i\neq 0, y_i\neq 0$ for all $j\in I\sm J$, thus
%the restriction of  $p_j$ to $Y_J^0$ is the projection
%$(\B{A}^2\sm\{(0,0)\})^{I-J}\times(\B{P}^1)^{J}\to
%(\B{A}^2\sm\{(0,0)\})^{I-J}$.
\end{Emp}

\begin{Emp} \label{E:loc}
{\bf Local coordinates.} Suppose that we are in the situation of \re{pink}.

(a) For every $J\subset I$, we set $\C{K}_J:=\prod_{j\in J}\C{K}_j$, and denote by $\wt{Y}_J$ the blow-up $\Bl_{\C{K}_J}(Y)$. Then we have a natural projection $\wt{Y}\to \wt{Y}_J$, which
is an isomorphism over $Y\sm(\cup_{i\in I\sm J}Y_i)$.

(b) Let $a\in X_I\subset X$ be a closed point.  Then there exists an open neighbourhood $U\subset X$ of $a$ and regular functions $\{\psi_i\}_{i\in I}$ on $U$ such that $\psi=(\psi_i)_{i\in I}$ is a smooth morphism $U\to\B{A}^I$, and $X_i\cap U$ is the scheme of zeros $Z(\psi_i)$ of $\psi_i$ for every $i\in I$.
Then $X_i\cap U$ is the schematic preimage $\psi^{-1}(Z(x_i))$ for all $i\in I$.

(c) Let  $a,b\in X_I$ be two closed points, and let $\psi_a:U_a\to \B{A}^I$
and  $\psi_b:U_b\to \B{A}^I$ be two smooth morphisms as in (b). Then $\psi:=(\psi_a,\psi_b)$
is a smooth morphism $U:=U_a\times U_b\to (\B{A}^2)^I$, which induces an isomorphism
$\wt{Y}\times_Y U\to(\wt{\B{A}^2})^I\times_{(\B{A}^{2})^I}U$.
\end{Emp}

\begin{Lem} \label{L:smooth}
In the situation of \re{pink}, for every $J\subset I$ the closed subscheme
$E_J\subset\wt{Y}$ is smooth of dimension $2d-|J|$.
\end{Lem}
\begin{proof}
In the basic case
\re{basic}, the assertion follows from the explicit description in \re{basic}(c) and the observation that $\wt{\B{A}}^2$ is smooth of dimension two. Since the assertion is
local on $Y$, the general case follows from this and \re{loc}. Namely, it suffices to show that for every closed point $a\in Y$ there exists an open neighbourhood $U$ such that $\pi^{-1}(U)\cap E_J$ is smooth of dimension $2d-|J|$.

Choose $J\subset I$ such that $a\in Y_J^0$. Then $a\in Y\sm (\cup_{i\in I\sm J}Y_i)$,
thus it follows from \re{loc}(a), that we can replace $I$ by $J$, thus assuming that
$a\in Y_I=X_I\times X_I$. In this case, by \re{loc}(c), there exists an open neighbourhood $U\subset Y$ of $a$ and a smooth morphism $U\to (\B{A}^2)^I$, which induces an isomorphism $\wt{Y}\times_Y U\to(\wt{\B{A}^2})^I\times_{(\B{A}^{2})^I}U$. Thus the assertion in general follows from the basic case.
\end{proof}

\begin{Emp} \label{E:prpreim}
{\bf Notation.} For every morphism $c:C\to Y$, we denote by
$\wt{c}:\wt{C}\to\wt{Y}$ the {\em strict preimage} of $c$. Explicitly,
$\wt{C}$ is the schematic closure of $c^{-1}(Y^0)\times_Y
\wt{Y}$ in $C\times_Y\wt{Y}$. Notice that $\wt{c}$ is a closed
embedding (resp. finite),  if $c$ is such.
%in which case we say that $\wt{c}(\wt{C})\subset\wt{Y}$ is a strict preimage of $c(C)\subset Y$.
\end{Emp}

\begin{Emp} \label{E:pinkfq}
{\bf Set-up.} In the situation of \re{pink}, assume that $k=\B{F}$ and that $X$ and all the $X_i$'s are defined over $\fq$.
\end{Emp}

\begin{Lem} \label{L:trans}
In the situation of \re{pinkfq}, assume that the correspondence
$c:C\to X\times X$ is locally contracting near $Z:=\p X$ over $\fq$. Let
$\Gm\subset X\times X$ be the graph of $\phi_q$, and let
$\wt{\Gm}\subset \wt{Y}$ be the strict preimage of $\Gm$. Then  $\wt{c}(\wt{C})\cap\wt{\Gm}\subset\pi^{-1}(X^0\times X^0)$.
\end{Lem}

\begin{proof}
%First we claim that we can assume that we are going to basic case \re{basic}, and that $c$ is contracting near $Z$.
Choose a closed point $\wt{y}\in \wt{c}(\wt{C})\cap\wt{\Gm}$, and set $y=(y_1,y_2):=\pi(\wt{y})\in\Gm$. 
Then $y_2=\phi_{q^n}(y_1)$, and it remains to show that $y_1\in X^0$.

By assumption (see \re{finfld}(c)), there exists an open neighborhood $U$ of $y_1$ defined over $\fq$ such that $c|_U$ is contracting near $Z\cap U$. Then $y_2=\phi_{q^n}(y_1)\in U$, so we can replace $c$ by $c|_U$, thus assuming that the correspondence $c$ is contracting near $Z$.

Assume that $y_1\notin X^0$. Arguing as in the proof of \rl{smooth} and using \re{loc}(a),
we can assume that $y_1\in X_I$, thus also $y_2=\phi_{q^n}(y_1)\in X_I$. 
By \re{loc}(b), there exists an open neighborhood $U$ of $y_1$ and a smooth morphism $\psi:U\to\B{A}^I$ such that 
$\psi^{-1}(Z(x_i))=X_i$ for every $i\in I$. Moreover, we can assume that $U$ and $\psi$ are defined over $\fq$. 
Again $y_2=\phi_{q^n}(y_1)\in U$, thus replacing $c$ by $c|_U$ we can assume that $U=X$. 

Consider correspondence $d:=(\psi,\psi)\circ c:C\to\B{A}^I\times \B{A}^I$. Then equalities $\psi^{-1}(Z(x_i))=X_i$ imply that
$d$ is contracting near $\p\B{A}^I:=\cup_{i\in I}Z(x_i)$. Then the assertion
for $c, X$ and  $X_i$ follows from the corresponding assertion for $d, \B{A}^I$ and  $Z(x_i)$.
In other words, we can assume that we are in the basic case \re{basic}.

We denote by  $\wt{Y}'\subset\wt{Y}$ the open subscheme, given by inequalities
$a_i\neq 0$ for all $i$. The assertion now follows from the part (a) of the following claim.
\end{proof}

\begin{Cl} \label{C:basic}
(a) We have inclusions $\wt{\Gm}\subset\wt{Y}'$, and $\wt{c}(\wt{C})\cap \wt{Y}'\subset\pi^{-1}(X^0\times
X^0)$.

(b) The projection $\wt{\Gm}\to\Gm$, induced by $\pi$, is an isomorphism.
\end{Cl}

\begin{proof}
%Abusing the notation we will identify regular functions $x_i$ and $y_i$ on $Y$ with their pullbacks to $\wt{Y}'$.
Notice that every $\ov{b}_i:=b_i/a_i$ is a regular function on $\wt{Y}'$, and that $\wt{Y}'$ is a closed subscheme of the affine space
$(\B{A}^3)^I$ with coordinates $x_i,y_i,\ov{b}_i$, given by equations $y_i=x_i\ov{b}_i$.

Consider the closed subscheme $\wt{\Gm}'\subset\wt{Y}'$ given
by equations $y_i=x_i^q$ and $\ov{b}_i=x_i^{q-1}$. Then $\pi:\wt{Y}\to Y$
induces an isomorphism $\wt{\Gm}'\to\Gm$. In particular, the projection $\wt{\Gm}'\to\Gm$ is proper.
Since $\Gm$ is closed in $Y$ and $\pi$ is proper,
this implies that $\wt{\Gm}'$ is closed in $\wt{Y}$ and therefore $\wt{\Gm}'$ has to coincide
with $\wt{\Gm}$. This implies assertion (b) and proves the inclusion $\wt{\Gm}\subset\wt{Y}'$.

Set $\wt{C}':=\wt{c}^{-1}(\wt{Y}')\subset\wt{C}$. We have to show that $\wt{C}'\subset\wt{c}^{-1}(\pi^{-1}(X^0\times X^0))$.
Consider regular functions $x:=\prod_{i\in I}x_i, y:=\prod_{i\in I}y_i$ and  $\ov{b}:=\prod_{i\in I}\ov{b}_i$
on $\wt{Y}'$. It remains to show that both pullbacks $\wt{c}^{\cdot}(x)$ and $\wt{c}^{\cdot}(y)$ to $\wt{C}'$ are invertible.

Since $c$ is contracting near $\p X$, there exist $n>0$ and a regular function $f$ on $C$
such that $c^{\cdot}(x)^n=c^{\cdot}(y)^{n+1}\cdot f$. Taking pullbacks to $\wt{C}'$, we get the equality
\begin{equation} \label{Eq:contr0}
\wt{c}^{\cdot}(x)^n=\wt{c}^{\cdot}(y)^{n+1}\cdot \wt{f},
\end{equation}
so it remains to show that $\wt{c}^{\cdot}(x)$ is invertible on $\wt{C}'$.
Since on $\wt{Y}'$ we have the equality $y=x\cdot\ov{b}$, equation \form{contr0} can be rewritten as
\begin{equation} \label{Eq:contr}
\wt{c}^{\cdot}(x)^n=\wt{c}^{\cdot}(x)^{n+1}\cdot \wt{c}^{\cdot}(\ov{b})^{n+1}\cdot\wt{f}.
\end{equation}
We claim that $\wt{c}^{\cdot}(x)\cdot \wt{c}^{\cdot}(\ov{b})^{n+1}\cdot\wt{f}=1$ on $\wt{C}'$, which obviously implies that
$\wt{c}^{\cdot}(x)$ is invertible.

Since $\wt{C}$ is the closure of $c^{-1}(Y^0)\times_Y \wt{Y}=C\times_Y (\wt{Y}\times_Y Y^0)$, we conclude that
$\wt{C}'$ is the closure of $C\times_Y (\wt{Y}'\times_Y Y^0)\subset C\times_Y\wt{Y}'$. Thus
it suffices to show that $\wt{c}^{\cdot}(x)\cdot \wt{c}^{\cdot}(\ov{b}^{n+1})\cdot\wt{f}=1$ on $C\times_Y (\wt{Y}'\times_Y Y^0)$.
By \form{contr}, it is enough to show that $x$ is invertible on $\wt{Y}'\times_Y Y^0$, that is,
$x_i(a)\neq 0$ for every closed point  $a\in\wt{Y}'\times_Y Y^0$ and every $i\in I$.
Since $y_i=x_i\ov{b}_i$, it suffices to show that for every $i\in I$ we have either $x_i(a)\neq 0$ or $y_i(a)\neq 0$, but this follows from the definition of $Y^0$.
\end{proof}

\begin{Lem} \label{L:grfrob}
In the situation of \re{pinkfq}, let $J\subset I$, and let $\Gm_J\subset X_J\times X_J$
(resp. $\Gm^0_J\subset X^0_J\times X_J^0$ ) be the graph of
$\phi_{q}$.

Then the schematic closure $\ov{\pi_J^{-1}(\Gm^0_J)}\subset E_J$ is smooth of dimension $d$, and the
schematic preimage $\pi_J^{-1}(\Gm_J)\subset E_J$ is a schematic
union of $\ov{\pi_{J'}^{-1}(\Gm_{J'})}$ with $J'\supset J$.
\end{Lem}

\begin{proof}
Assume first that we are in the basic case. By \re{basic}(c),
we immediately reduce to the case $J=\emptyset$. In this case, $\ov{\pi_J^{-1}(\Gm^0)}=\wt{\Gm}$ is isomorphic to $\Gm$ (by \rcl{basic}(b)), thus it is smooth of dimension $d$. This shows the first assertion.

Next, the second assertion immediately reduces to the case $m=1$, in which it is easy. Indeed,
$\Gm\subset \B{A}^2$ is given by equation $y=x^q$, thus $\pi^{-1}(\Gm)\subset \B{A}^2\times \B{P}^1$ is given by equations $y=x^q$ and $xb=x^q a$.
Thus $\pi^{-1}(\Gm)$ equals the schematic union of $\wt{\Gm}$, given by $y=x^q$ and $b=x^{q-1}a$, and the exceptional divisor $x=y=0$.

Finally, arguing as in Lemmas \ref{L:smooth} and \ref{L:trans},  we reduce the general case to the basic case.
Namely, replacing $X$ by its open subset, we may assume that there exists a smooth map
$\psi:X\to\B{A}^I$ defined over $\fq$ such that $\psi^{-1}(Z(x_i))=X_i$ for every $i\in I$.
Then $\psi$ induces an isomorphism $\wt{Y}\isom Y\times_{(\B{A}^2)^I}(\wt{\B{A}^2})^I$.
Moreover, if we denote by $(\cdot)_{\B{A}^I}$ the objects corresponding to $\B{A}^I$ instead of $X$, then
$\psi$ induces a smooth morphism
$\Gm_{J}\to(\Gm_{J})_{\B{A}^I}$, hence smooth morphisms $\pi_J^{-1}(\Gm_J)\to \pi_J^{-1}(\Gm_J)_{\B{A}^I}$ and $\pi_J^{-1}(\Gm^0_J)\to \pi_J^{-1}(\Gm^0_J)_{\B{A}^I}$. Thus both assertions follow from the basic case.
\end{proof}

\section{Formula for the intersection number}

\begin{Emp}
{\bf Intersection theory.} (a) Let $X$ be a scheme of finite type over a field $k$. For  $i\in\B{N}$, we denote by $A_i(X)$ the group of $i$-cycles modulo rational equivalence (see \cite[1.3]{Fu}). For a closed subscheme $Z\subset X$ of pure dimension $i$, we denote by $[Z]$ its class in $A_i(X)$ (see \cite[1.5]{Fu}). For a proper morphism of schemes $f:X\to Y$,
we denote by $f_*:A_i(X)\to A_i(Y)$ the induced morphism (see \cite[1.4]{Fu}).
In particular, if $X$ is proper over $k$, then the projection $p_X: X\to \Spec k$ gives rise to the
degree map $\deg:=(p_X)_*:A_0(X)\to A_0(\Spec k)=\B{Z}$.

(b) Let $X$ be a smooth connected scheme
over $k$ of dimension $d$. For every $i$,  we set
$A^i(X):=A_{d-i}(X)$. For every $i,j$, we have the intersection
product \\ $\cap: A^i(X)\times A^j(X)\to A^{i+j}(X)$ (see
\cite[8.3]{Fu}). In particular, if $X$ is also proper,  we have the
intersection product $\cdot:=\deg\circ\cap:A_i(X)\times
A^i(X)\to\B{Z}$.

(c) For every morphism $f:X\to Y$ between smooth connected schemes, we
have a pullback map $f^*:A^i(Y)\to A^i(X)$ (see \cite[8.1]{Fu}). Moreover, if $f$ is
proper, then we have the equality $f^*(x)\cdot
y=x\cdot f_*(y)$ for every $x\in A^i(Y)$ and $y\in A_i(X)$, called
the projection formula (see \cite[Prop 8.3(c)]{Fu}).
\end{Emp}

\begin{Emp} \label{E:expb}
{\bf Example.} Let $f:X\to Y$ be a morphism between smooth connected schemes,
and let $C\subset Y$ be a closed subscheme such that both
inclusions $i:C\hra Y$ and $i':f^{-1}(C):=C\times_Y X\hra X$ are
regular imbeddings of codimension $d$. Then
$f^*([C])=[f^{-1}(C)]$.

Indeed, $f$ is a composition $X\overset{(\Id,f)}{\lra}X\times Y\overset{\pr_2}{\lra}Y$.
Since the assertion for $\pr_2$ is clear, and $(\Id,f)$ is a regular embedding,
we may assume that $f$ is a regular embedding. In this case, the
induced morphism $f^{-1}(C)\to C$ is a regular embedding as well,
so the assertion follows, for example, from \cite[Thm 6.2 (a) and Rem
6.2.1]{Fu}.
\end{Emp}

From now on we assume that $k$ is an algebraically closed field, and $l$ is a prime,
different from the characteristic of $k$.

\begin{Emp} \label{E:cycle}
{\bf The cycle map.} Let $X$ be a $d$-dimensional smooth connected scheme over $k$.

(a) Recall (see \cite{Gr}) that for every closed integral
subscheme $C\subset X$ of codimension $i$ one can associate its
class $\cl(C)\in H^{2i}(X,\ql(i))$. Namely, by the
Poincar\'e duality, $\cl(C)$ corresponds to the composition
\[
H^{2(d-i)}_c(X,\ql(d-i))\overset{\res_C}{\lra} H_c^{2(d-i)}(C,\ql(d-i))\overset{\Tr_{C/k}}{\lra} \ql
\]
of the restriction map and the trace map.
In other words, $\cl(C)\in H^{2i}(X,\ql(i))$ is characterized by the condition that
$\Tr_{X/k}(x\cup \cl(C))= \Tr_{C/k}(\res_C^*(x))$ for every $x\in H^{2(d-i)}_c(X,\ql(d-i))$.

(b) Note that $\cl(C)$ only depends on the class $[C]\in A^i(X)$ (see, for example,
\cite[Thm 6.3]{Lau}), thus $\cl$ induces a map $A^i(X)\to H^{2i}(X,\ql(i))$.
Furthermore, for every $x\in A^i(X)$ and $y\in A^j(X)$, we have an
equality $\cl(x\cap y)=\cl(x)\cup\cl(y)$ (see \cite[Cor 7.2.1]{Lau} when $X$ is quasi-projective, which suffices for the purpose of this note, or use \cite[Lem 2.1.2]{KS} in the general case).

(c) If $X$ is proper, then it follows from the
description of (a) that $\Tr_{X/k}(\cl(x))=\deg(x)$ for every
$x\in A_0(X)$.
\end{Emp}

\begin{Emp} \label{E:end}
{\bf Endomorphism of the cohomology.} Let $X_1$ and $X_2$ be smooth connected
proper schemes over $k$ of dimensions $d_1$ and $d_2$,
respectively, and set $Y:=X_1\times X_2$.

(a) Fix an element $u\in H^{2d_1}(Y,\ql(d_1))$. Then $u$ induces
a morphism $H^i(u):H^i(X_1,\ql)\to H^i(X_2,\ql)$ for every $i$.
Namely, by the Poincar\'e duality, $H^i(u)$ corresponds to the map
$H^i(X_1,\ql)\times H^{2d_2-i}(X_2,\ql(d_2))\to\ql$, which sends
$(x,y)$ to $\Tr_{Y/k}(u\cup(x\boxtimes y))$. Here we set $x\boxtimes y:=p_1^*x\cup p_2^*y\in
H^{2d_2}(X_2,\ql(d_2))$.

(b) As in (a), an element $v\in H^{2d_2}(Y,\ql(d_2))$ induces a morphism
$H^i(v):H^i(X_2,\ql)\to H^i(X_1,\ql)$. Moreover, we have $u\cup v\in H^{2(d_1+d_2)}(Y,\ql(d_1+d_2))$,
and  it was shown in \cite[Prop. 3.3]{Gr} that $\Tr_{Y/k}(u\cup v)$
equals the alternate trace
\[
\Tr(H^*(v)\circ
H^*(u)):=\sum_{i}(-1)^i\Tr(H^i(v)\circ H^i(u)).
\]
\end{Emp}

\begin{Emp} \label{E:Excycl}
{\bf Connection with the cycle map.} In the situation of \re{end}, let $C\subset
X_1\times X_2$ be a closed integral subscheme of dimension $d_2$, hence of codimension $d_1$.

(a) Then $C$ gives rise to an element $\cl(C)\in   H^{2d_1}(Y,\ql(d_1))$, hence it induces
an endomorphism $H^i([C]):=H^i(\cl(C)):H^i(X_1,\ql)\to H^i(X_2,\ql)$.
Let $(p_1,p_2):C\hra X_1\times X_2$ be the inclusion.
Denote by $(p_2)_*: H^i(C,\ql)\to H^i(X_2,\ql)$ the push-forward
map, corresponding by duality to the map $ H^i(C,\ql)\times
H^{2d_2-i}(X_2,\ql(d_2))\to\ql$, defined by  $(x,y)\mapsto
\Tr_{C/k}(x\cup p_2^*(y))$.

(b) We claim that the endomorphism
$H^i([C]):H^i(X_1,\ql)\to H^i(X_2,\ql)$ decomposes as  $H^i(X_1,\ql)\overset{(p_1)^*}{\lra}
H^i(C,\ql)\overset{(p_2)_*}{\lra}H^i(X_2,\ql)$. Indeed, by
\re{cycle}(a) and \re{end}(a), $H^i([C])$ corresponds by
duality to the map $H^i(X_1,\ql)\times H^{2d_2-i}(X_2,\ql(d_2))\to\ql$, defined by the rule
$(x,y)\mapsto \Tr_{Y/k}((x\boxtimes y)\cup\cl(C))=\Tr_{C/k}(p_1^*(x)\cup p_2^*(y))$. Thus
$H^i([C])$ decomposes as $(p_2)_*\circ (p_1)^*$.

(c) It follows from (b) that if $\Gm_f\subset X_1\times X_2$ is the
graph of the map $f:X_2\to X_1$, then $H^i([\Gm_f])$ is simply the
pullback map $f^*=H^i(f)$.

(d) Assume now that $X_1=X_2=X$ is of dimension $d$ and that both
$p_1,p_2:C\to X$ are dominant. Then $p_1$ is generically finite, and
it follows from the description of (b) and basic properties of the
trace map that $H^{2d}([C])=\deg(p_1)\Id$.
\end{Emp}

\begin{Emp} \label{E:Ex}
{\bf Example.}
In the situation of \re{end}, let $C\subset
X_1\times X_2$ be a closed integral subscheme of dimension $d_1$, and let
$f:X_2\to X_1$ be a morphism. Then $C$ defines a class $[C]\in A_{d_1}(Y)=A^{d_2}(Y)$ and a morphism
$H^i([C]):H^i(X_2,\ql)\to H^i(X_1,\ql)$ (see \re{Excycl}(a)). 

On the other hand, morphism $f$ defines a class
$[\Gm_f]\in A^{d_1}(Y)$, and
a morphism $f^*=H^i(f):H^i(X_1,\ql)\to H^i(X_2,\ql)$. Then have the equality
\begin{equation} \label{Eq:traces}
[C]\cdot[\Gm_f]=\Tr(f^*\circ H^*([C])).
\end{equation}
Indeed, since $f^*=H^*([\Gm_f])$ by \re{Excycl}(c), the right hand side of
\form{traces} equals $\Tr_{Y/k}(\cl(C)\cup\cl(\Gm_f))$ by \re{end}(b), hence to
$\Tr_{Y/k}(\cl([C]\cap[\Gm_f]))=[C]\cdot[\Gm_f]$ by \re{cycle}(b),(c).
\end{Emp}

\begin{Emp} \label{E:purity}
{\bf Purity.} Let $X$ be a smooth proper variety of dimension $d$ over $\B{F}$, defined over $\fq$.
 It is well known that $H^i(X,\ql)=0$ if $i>2d$. Moreover, by a theorem of Deligne \cite{De}, for every $i$ and
every embedding $\iota:\qlbar\hra\B{C}$,
all eigenvalues $\la$ of $\phi_q^*:H^i(X,\ql)\to H^i(X,\ql)$ satisfy $|\iota(\la)|=q^{i/2}$.
\end{Emp}

\begin{Emp} \label{E:not}
{\bf Notation.} From now on we assume that we are in the situation
of \re{pink}.  

Let $C\subset Y$ be a closed integral subscheme of dimension $d$ such that $C\cap Y^0\neq\emptyset$, and
let $\wt{C}\subset\wt{Y}$ be the strict transform of $C$.

(a) Consider cycle classes $[C]\in A_d(Y)$ and $[\wt{C}]\in A_d(\wt{Y})$. For every
$J\subset I$, schemes $Y_J$ and $E_J$ are smooth and proper over $k$, therefore we can form a cycle class $[\wt{C}]_J:=(\pi_{J})_*i_J^*[\wt{C}] \in A_{d-|J|}(Y_{J})$. In particular, we have $[\wt{C}]_J=0$,
if $X_{J}=\emptyset$.

(b) Assume in addition that $k=\B{F}$, and that $X$ and all the $X_i$'s are defined over $\fq$.

Fix $n\geq 0$. Let $\Gm=\Gm_{q^n}\subset Y=X\times X$,  $\Gm_J=\Gm_{J,q^n}\subset Y_J=X_J\times X_J$ and $\Gm^0_J\subset X^0_J\times X^0_J$ be the graphs of $\phi_{q^n}$, and let $\wt{\Gm}=\wt{\Gm}_{q^n}\subset\wt{Y}$ be the strict preimage of $\Gm$.
We denote by  $[\wt{\Gm}]\in A_d(\wt{Y})$,  $[\Gm]\in A_d(Y)$
and $[\Gm_J]\in A_{d-|J|}(Y_J)$ the
corresponding classes.
\end{Emp}

The following result is an analog of \cite[Prop.
IV.6]{Laf}.

\begin{Lem} \label{L:inters}
In the situation of \re{not}, we have the equality

\begin{equation} \label{Eq:classes}
[\wt{C}]\cdot[\wt{\Gm}]=[C]\cdot[\Gm]+
\sum_{J\neq\emptyset}(-1)^{|J|}[\wt{C}]_J\cdot[\Gm_J].
\end{equation}
\end{Lem}

\begin{proof}
Note that $[C]=\pi_*[\wt{C}]$. Since $[\wt{C}]_J=(\pi_J)_*i^*_J
[\wt{C}]$, we conclude from the projection formula that $[C]\cdot
[\Gm]=[\wt{C}]\cdot\pi^*[\Gm]$ and $[\wt{C}]_J\cdot
[\Gm_J]=[\wt{C}]\cdot (i_J)_*\pi_{J}^*[\Gm_J]$. Hence it suffices
to show that
\begin{equation} \label{Eq:1}
[\wt{\Gm}]=\pi^*[\Gm]+\sum_{J\neq\emptyset}(-1)^{|J|}(i_J)_*\pi_{J}^*
[\Gm_J].
\end{equation}
Note that for every $J$, scheme $X_J$ is smooth, thus the inclusion  $\Gm_J\hra Y_J$ is regular of codimension $\dim X_J=d-|J|$. On the other hand, by \rl{grfrob}, every schematic irreducible component of $\pi_J^{-1}(\Gm_J)$ is of dimension $d-|J|$, thus the inclusion
$\pi_J^{-1}(\Gm_J)\hra E_J$ is also regular of codimension $d-|J|$.
It therefore follows from example \re{expb} that $\pi_{J}^*
[\Gm_J]=[\pi_{J}^{-1}(\Gm_J)]$.

Using \rl{grfrob} again, we get that
$[\pi_{J}^{-1}(\Gm_J)]=\sum_{J'\supset J}
[\ov{\pi_{J'}^{-1}(\Gm^0_{J'})}]$, thus
\begin{equation} \label{Eq:2}
(i_J)_*\pi_{J}^*[\Gm_J]=\sum_{J'\supset J}
[\ov{\pi_{J'}^{-1}(\Gm^0_{J'})}].
\end{equation}
Applying this in the case $J=\emptyset$, we get the equality
\begin{equation} \label{Eq:3}
\pi^*[\Gm]=[\wt{\Gm}]+\sum_{J'\neq\emptyset}
[\ov{\pi_{J'}^{-1}(\Gm^0_{J'})}].
\end{equation}
Finally, equation \form{1} follows from \form{2}, \form{3} and the observation that for every $J'\neq 0$,
we have the equality $\sum_{J\subset J'}(-1)^{|J|}=0$.
\end{proof}

%\begin{Rem}
%If  a correspondence $c:Y\to X\times X$  is contracting near a
%closed subset $Z\subset X$ in a neighborhood of fixed points, then
%$Z$ is $c$-invariant in a neighborhood of fixed points. To give a
%partial converse we need the following definition.
%\end{Rem}

\begin{Cor} \label{C:est}
In the situation of \rl{inters}, assume that both projections
$p_1,p_2:C\to X$ are dominant. Then for every sufficiently large
$n$, we have $[\wt{C}]\cdot[\wt{\Gm}_{q^n}]\neq 0$, thus
$\wt{C}\cap\wt{\Gm}_{q^n}\neq\emptyset$.
\end{Cor}

\begin{proof}
By equation \form{traces} from \re{Ex}, we have 
$[C]\cdot[\Gm_{q^n}]=\Tr((\phi^*_q)^n\circ H^*([C]), H^*(X,\ql))$ and
$[\wt{C}]_J\cdot[\Gm_{J,q^n}]=\Tr((\phi^*_q)^n\circ H^*([\wt{C}]_J), H^*(X_J,\ql))$ for all $n$ and $J$.

By \re{Excycl}(d), we have
$H^{2d}([C])=\deg(p_1)\Id$. Therefore we conclude from a combination of Deligne's theorem (see \re{purity}) and \rl{bound} below that for large $n$ we have
$[C]\cdot[\Gm_{q^n}]\sim\deg(p_1)q^{dn}$, and
$[\wt{C}]_J\cdot[\Gm_{J,q^n}]= O(q^{2(d-|J|)n})$ for $J\neq\emptyset$.
Hence, by \form{classes}, for large $n$ we have
$[\wt{C}]\cdot[\wt{\Gm}_{q^n}]\sim\deg(p_1)q^{dn}$, thus  $[\wt{C}]\cdot[\wt{\Gm}_{q^n}]\neq 0$.
\end{proof}

\begin{Lem} \label{L:bound}
Let $a>1$, and let $A,B\in\Mat_d(\B{C})$ be such that every eigenvalue $\la$ of $A$ satisfies $|\la|\leq a$.
Then $\Tr(A^n B)$ is of magnitude $O(n^{d-1}a^n)$.
\end{Lem}

\begin{proof}
We can assume that $A$ has a Jordan form. Then all entries of $A^n$ are of magnitude $O(n^{d-1}a^n)$.
This implies the assertion.
\end{proof}

\section{Proof of the main theorem}

\begin{Emp} \label{E:red}
{\bf Reduction steps.} (a) Assume that we are given a commutative
diagram of schemes of finite type over $\B{F}$
\begin{equation} \label{Eq:red}
\CD
        \wt{C}^0        @>{\wt{c}^0}>>         \wt{X}^0\times \wt{X}^0\\
        @V{f_C}VV                        @V{f\times f}VV\\
       C^0        @>{c^0}>>         X^0\times X^0,
\endCD
\end{equation}
such that $c^0$ and $\wt{c}^0$ satisfy the assumptions of \rt{main}, and $f:\wt{X}\to X$ is defined over $\fq$. Then
\rt{main} for $\wt{c}^0$ implies that for $c^0$.

Indeed, for every $n\geq 0$,
we have an inclusion $f_C((\wt{c}^0)^{-1}(\Gm^0_{q^n}))\subset
(c^0)^{-1}(\Gm^0_{q^n})$. Thus we have $(c^0)^{-1}(\Gm^0_{q^n})\neq\emptyset$, if
$(\wt{c}^0)^{-1}(\Gm^0_{q^n})\neq\emptyset$.

(b) By (a), for every open subset $U\subset X^0$ over $\fq$,
\rt{main} for $c^0$ follows from that for $c^0|_{U}$. Also for
every open $W\subset C^0$, \rt{main} for $c^0$ follows from that for
$c^0|_W$.

(c) Assume that the commutative diagram \form{red} satisfies
$f=\Id_{X^0}$ and $f_C$ is surjective. Then
$f_C((\wt{c}^0)^{-1}(\Gm^0_{q^n}))=(c^0)^{-1}(\Gm^0_{q^n})$, thus \rt{main}
for $c^0$ implies that for $\wt{c}^0$.

(d) For every $m\in\B{N}$, the Frobenius twist
$(c^0)^{(m)}:C^0\to X^0\times X^0$ (see \re{finfld}(a)) satisfies
$((c^0)^{(m)})^{-1}(\Gm^0_{q^n})=(c^0)^{-1}(\Gm^0_{q^{n+m}})$. Thus
\rt{main} for $c^0$ follows from that for $(c^0)^{(m)}$.

(e) Fix $r\in\B{N}$. Then every $n\in\B{N}$ has a form $n=rn'+m$ with $m\in\{0,\ldots, r-1\}$.
Thus the assertion \rt{main} for $c^0$ over $\fq$ is equivalent to the corresponding assertion for
$(c^0)^{(m)}$ over $\B{F}_{q^r}$ for $m=0,\ldots, r-1$.
\end{Emp}

\begin{Cl} \label{C:red}
It is enough to prove \rt{main} under the assumption that there exists
a Cartesian diagram
\begin{equation} \label{Eq:cred}
\CD
        C^0        @>{c^0}>>         X^0\times X^0\\
        @V{j_C}VV                        @V{j\times j}VV\\
       C        @>{c}>>         X\times X
\endCD
\end{equation}
of schemes of finite type over $\B{F}$ such that

(i) $X$ is irreducible projective, and $j$ is an open embedding, both 
defined over $\fq$.

(ii) $C$ is irreducible of dimension $\dim X$, and  $c$ is finite.

(iii) $X$ is smooth, and the complement $\p X:=X\sm j(X^0)$ is a
union of smooth divisors $X_i$ with normal crossings, defined over
$\fq$.

(iv)  The correspondence $c$ is locally contracting near $\p X$  over $\fq$.
\end{Cl}

\begin{proof}
We carry out the proof in six steps.

{\bf Step 1.} We may assume that $X^0$ is quasi-projective, and that
$\dim C^0=\dim X^0$.

Indeed, by \re{red}(b), we may replace $X^0$ and $C^0$ by their open neighborhoods,
thus assuming that $X^0$ and $C^0$ are affine. Next, since $c^0_1,c^0_2:C^0\to X^0$ are dominant, the
difference $r:=\dim C^0-\dim X^0$ is nonnegative, and there exist
maps $d^0_1,d^0_2:C^0\to\B{A}^r$ such that $e^0_i:=c^0_i\times
d^0_i:C^0\to X^0\times\B{A}^r$ is dominant for $i=1,2$. Thus by \re{red}(a), it
suffices to prove the theorem for $e^0:=(e^0_1,e^0_2)$ instead of
$c^0$, thus we may assume that $\dim C^0=\dim X^0$.

{\bf Step 2.} In addition to the assumptions of Step 1, we may assume that $c^0$ is a closed embedding.

Indeed, assume that we are in the situation of Step 1, and let $C'$ be the closure $\ov{c^0(C^0)}\subset X^0\times X^0$.
Then the image $c^0(C^0)$ contains a non-empty open subset $U\subset C'$. Then by
\re{red}(b), (c), we may replace $c^0$ by the inclusion $U\hra
X^0\times X^0$, thus assuming that $c^0$ is a locally closed
embedding. Moreover, since $\dim(C'\sm U)<\dim C^0=\dim X^0$, we can replace
$X^0$ by $X'_0:=X^0\sm (\ov{c^0_1(C'\sm U)}\cup \ov{c^0_2(C'\sm U)})$
and $c^0$ by $c^0|_{X'_0}$, thus assuming that $c^0$ is a closed
embedding.

{\bf Step 3.} We may assume that there exists a Cartesian diagram
\form{cred} satisfying (i) and (ii).

Indeed, assume that we are in the situation of Step 2.
Since $X^0$ is quasi-projective, there exists an open
embedding with dense image $j:X^0\hra X$ over $\fq$ with $X$
projective.  Let $C\subset X\times X$ be the closure of
$c^0(C^0)$, and let $c:C\hra X\times X$ be the inclusion map.
Then $c$ defines a  Cartesian diagram
\form{cred} satisfying (i) and (ii).

{\bf Step 4.} In addition to the assumptions of Step 3, we may assume that $\p X:=X\sm j(X_0)$
is locally $c$-invariant over $\fq$.

Indeed, assume that we are in the situation of Step 3. Then, by \rco{blowup}, 
there exists an open subset $V\subset X^0$ and a blow-up $\pi:\wt{X}\to X$, which is an isomorphism over $V$,
such that both $V$ and $\pi$ are defined over $\fq$ and for every
map $\wt{c}:\wt{C}\to \wt{X}\times \wt{X}$ lifting $c$, the closed subset $\wt{X}\sm
\pi^{-1}(V)\subset \wt{X}$ is locally $\wt{c}$-invariant over $\fq$.

Replacing $X^0$ by $V$ and $c^0$ by $c^0|_V$, we may assume that
$V=X^0$ (use \re{red}(b)). Let $\wt{j}$ be the inclusion $X^0\cong \pi^{-1}(X^0)\hra
\wt{X}$, let  $\wt{C}\subset \wt{X}\times \wt{X}$ be the closure of
$C^0\subset C\times_{(X\times X)}(\wt{X}\times \wt{X})$, and let
$\wt{c}:\wt{C}\to \wt{X}\times \wt{X}$ be the projection. Then, replacing $c$ by
$\wt{c}$, we get the Cartesian diagram we are looking for.

{\bf Step 5.} In addition to the assumptions of Step 4, we may assume that property (iii) is
satisfied.

Indeed, assume that we are in the situation of Step 4.
By a theorem of de Jong on alterations (see \cite[Thm 4.1 and Rem 4.2]{dJ}),
there exists a proper generically finite map $\pi:\wt{X}\to X$ such that $\wt{X}$ is smooth
and geometrically connected over $\B{F}$, and $\pi^{-1}(\p X)\subset \wt{X}$ is a union of smooth divisors
with strict normal crossings $\wt{X}_i$. Choose $r$ such that $\wt{X}$, $\wt{X}_i$ and $\pi$ are defined over
$\B{F}_{q^r}$.

By \re{red}(e), in order to prove \rt{main} for $c^0$ over $\fq$ it suffices to prove
\rt{main} over $\B{F}_{q^r}$ for $(c^0)^{(m)}$ for $m=0,\ldots, r-1$. Note that every $(c^0)^{(m)}$ satisfies
all the assumptions of Step 4. Indeed, the twist $c^{(m)}:C\to X\times X$ is finite, because $c$ and $\phi_q:X\to X$ are finite, and $c^{(m)}$ gives rise to the Cartesian diagram
we are looking for.  Thus replacing $\fq$
by its finite extension $\B{F}_{q^r}$ and $c^0$ by $(c^0)^{(m)}$, we may assume
that $\wt{X}$, $\wt{X}_i$ and $\pi$ are defined over $\fq$.

Since $c_1, c_2:C\to X$ are dominant, there exists a unique
irreducible component $\wt{C}$ of $C\times_{(X\times X)}(\wt{X}\times
\wt{X})$ such that both projections $\wt{c}_1,\wt{c}_2:\wt{C}\to\wt{X}$ are dominant.
Replacing $X^0$ by $\pi^{-1}(X^0)$ and $c$ by $\wt{c}=(\wt{c}_1,\wt{c}_2)$, we get the required Cartesian diagram.
Indeed, properties (i)-(iii) are satisfied by construction, while the locally invariance property of Step 4
is preserved by \rl{locinv}.

{\bf Step 6.} We may assume that all the assumptions of \rcl{red} are satisfied.

Assume that we are in the situation of Step 5. Choose $m\in\B{N}$
such that $q^m>\ram(c_2,\p X)$. Then  $\p X$
is locally $c$-invariant over $\fq$ by Step 4, then it is locally $c^{(m)}$-invariant, by \re{remfinfld}. 
Thus $c^{(m)}$ is locally contracting near $\p X$ by \rl{contr} over $\fq$. By \re{red}(d), we can replace
$c^0$ by $(c^0)^{(m)}$ (and $c$ by $c^{(m)}$), thus we can assume that all the assumptions of \rcl{red} are satisfied.
\end{proof}

\begin{Emp}
\begin{proof}[Proof of \rt{main}] By \rcl{red}, we can assume that
there exists a Cartesian diagram \form{cred} satisfying properties
(i)-(iv). To simplify the notation, we will identify $X^0$ with
$j(X^0)\subset X$. % $C$ with $c(C)\subset Y$ and $\wt{C}$ with
%$\wt{c}(\wt{C})\subset\wt{Y}$.
By property (iv) and \rl{trans}, we have an inclusion $\pi(\wt{c}(\wt{C})\cap\wt{\Gm}_{q^n})
\subset X^0\times X^0$ for every $n\in\B{N}$. Since $\pi_C(\wt{C})\subset C$ and  $\pi(\wt{\Gm}_{q^n})\subset\Gm_{q^n}$, we conclude that
\begin{equation} \label{Eq:incl}
\pi(\wt{c}(\wt{C})\cap\wt{\Gm}_{q^n})\subset (c(C)\cap
\Gm_{q^n})\cap (X^0\times X^0)=c^0(C^0)\cap  \Gm^0_{q^n}=c^0((c^0)^{-1}(\Gm^0_{q^n})).
\end{equation}
%
%Thus in order to show that $(c^0)^{-1}(\Gm^0_{q^n})\neq\emptyset$, it is enough to show
%that $\wt{c}(\wt{C})\cap\wt{\Gm}_{q^n}\neq\emptyset$.

On the other hand, since $c$ is finite its image  $C':=c(C)\subset Y$ is a closed integral subscheme of dimension $d$, and
its strict preimage $\wt{C'}\subset \wt{Y}$ equals $\wt{c}(\wt{C})$. Thus, it follows from \rco{est}, applied to $C'$, that for every
sufficiently large $n\in\B{N}$ we have
$\wt{c}(\wt{C})\cap\wt{\Gm}_{q^n}\neq\emptyset$. Thus, by \form{incl},
we get $(c^0)^{-1}(\Gm^0_{q^n})\neq\emptyset$, and the proof of \rt{main} is complete.
\end{proof}
\end{Emp}

\end{document}